\documentclass[12pt]{amsart}
\usepackage{amssymb}

\usepackage{graphicx}

\usepackage{amssymb}
\usepackage{amsmath}
\usepackage{amsthm}
\usepackage{times}
\usepackage{graphicx}

\usepackage{amssymb}
\usepackage{amsmath}
\usepackage{amsthm}
\usepackage{times}

\usepackage{amsthm}
\newtheorem{theorem}{Theorem}[section]
\newtheorem{lemma}[theorem]{Lemma}
\newtheorem{corollary}[theorem]{Corollary}
\newtheorem{proposition}[theorem]{Proposition}

\theoremstyle{definition}

\newtheorem{example}[theorem]{Example}

\theoremstyle{remark}
\newtheorem{remark}[theorem]{Remark}

\numberwithin{equation}{section}

\usepackage[
            pdftex,
           colorlinks=true,
            urlcolor=blue,       % \href{...}{...} external (URL)
            filecolor=blue,     % \href{...} local file
            linkcolor=blue,       % \ref{...} and \pageref{...}
            citecolor=blue,         % \cite{...}
            pdftitle= {Title},
            pdfauthor={Author},
            pdfsubject={Subject},
            pdfkeywords={Key}
            pagebackref,
            pdfpagemode=None,
            bookmarksopen=true]
            {hyperref}
\usepackage{hyperref}
\usepackage{color}

\usepackage[curve,all]{xy}
 \xyoption{curve}
 \xyoption{color}
 \xyoption{line}
\xyoption{arc}

\newcommand{\calM}{\mathcal{M}}

\newcommand{\calO}{\mathcal{O}}
\newcommand{\calP}{\mathcal{P}}
\newcommand{\calQ}{\mathcal{Q}}

\newcommand{\calS}{\mathcal{S}}

\newcommand{\sfS}{\mathsf{S}}
\newcommand{\sfP}{\mathsf{P}}
\newcommand{\sfD}{\mathsf{D}}

\newcommand{\frakr}{\mathfrak{r}}
\newcommand{\frakl}{\mathfrak{l}}

\newcommand{\frakS}{\mathfrak{S}}

\newcommand{\bbC}{\mathbb{C}}

\newcommand{\bbP}{\mathbb{P}}
\newcommand{\bbQ}{\mathbb{Q}}
\newcommand{\bbR}{\mathbb{R}}

\newcommand{\bbZ}{\mathbb{Z}}

\newcommand{\bfv}{\mathbf{v}}

\newcommand{\la}{\langle}
\newcommand{\ra}{\rangle}

\newcommand{\SL}{\textrm{SL}}

\newcommand{\Pic}{\text{Pic}}

\newcommand{\adj}{\text{adj}}

\newcommand{\End}{\text{End}}

\newcommand{\rank}{\text{rank}}

\newcommand{\da}{\dasharrow}
\newcommand{\Kum}{\textrm{Kum}}
\newcommand{\Jac}{\textrm{Jac}}
\newcommand{\SU}{\textrm{SU}}
\newcommand{\ns}{\textrm{ns}}
\newcommand{\pr}{\textrm{pr}}
\newcommand{\Bis}{\textrm{Bis}}

\newcommand{\beq}{\begin{equation}}
\newcommand{\eeq}{\end{equation}}

\begin{document}
\title{Corrado Segre and nodal cubic threefolds}
\author{Igor Dolgachev}
\begin{abstract} We discuss the work of Corrado Segre on nodal cubic hypersurfaces with emphasis on the cases of 6-nodal and 10-nodal cubics. In particular we discuss the Fano surface of lines and conic bundle structures on such threefolds.  We review some of the modern research in algebraic geometry related to Segre's work.
\end{abstract}

\address{Department of Mathematics, University of Michigan, 525 E. University Av., Ann Arbor, Mi, 49109, USA}
\email{idolga@umich.edu}

\maketitle
 
\section{Introduction}
The following is a  detailed exposition of my talk devoted to two memoirs of Corrado Segre on irreducible cubic hypersurfaces in $\bbP^4$ with $d$ ordinary double points (nodes) \cite{Segre1}, \cite{Segre2}. We will show in this article how Segre's work is related to the current research in algebraic geometry.
  
Let $X$ be an irreducible  cubic hypersurface in $\bbP^4$ with a node $q$. Choose a hyperplane $H$ that intersects $X$ transversally along a nonsingular cubic surface $F$ and consider the projective coordinates in $\bbP^4$ such that $H = V(t_0)$ and  $q$ is equal to the point $[1,0,0,0,0]$. Then the equation of $X$ can be written in the form
\begin{equation}\label{cubic1}
X:t_0a_2(t_1,t_2,t_3,t_4)+a_3(t_1,t_2,t_3,t_4) = 0,
\end{equation}
where $a_2$ and $a_3$ are homogeneous forms of degrees 2 and 3, respectively, such that $Q = V(a_2)$ is a nonsingular quadric surface in $H$ and $F = V(a_3)$.  The equations 
$a_2 = a_3= 0$ define a curve  of degree 6 in the 
hyperplane  $H$. Following \cite{Fink1} we call it the \emph{associated curve} of $X$ with respect to $q$, it will be denoted by $C(X,q)$.  The curve $C(X,q)$ is 
a curve of bidegree $(3,3)$ lying on $Q$. We will assume that it is reduced. Let $X'$ be the proper transform of $X$ under the blow-up of $\bbP^4$ at the point $q$. Then the projection map $\pr_q:X'\to \bbP^3$ defines an isomorphism between $X'$ and the  blow-up of $\bbP^3$ with center at $C(X,q)$. 
The inverse rational map 
$$\alpha:\bbP^3\dasharrow \bbP^4$$
is given by the linear system of cubic containing $C(X,q)$. The 
latter is spanned by the cubic $F = V(a_3)$ and any cubic of the form $V(a_2l)$, where $l$ is a linear form in $t_1,\ldots,t_4$. It follows that the rational map $\alpha$ is given by the formula
$(t_1,\ldots,t_4) \mapsto (a_3,t_1a_2,t_2a_2,t_3a_2,t_4a_2)$, and hence $Q$ 
is contracted to the point $q = [1,0,0,0,0]$. Also, it is clear that any singular point $q'\ne q$ of $X$ is projected to a singular point of $C(X,q)$. Since we assume that all singular points are ordinary double points, their images are ordinary double points of $C(X,q)$.

The arithmetical genus of a curve of bidegree $(3,3)$ on $Q$ is  equal to 4. If the curve has more than four double points, it must be reducible. A simple analysis 
shows that the largest possible number $k$ of double points of $C(X,q)$ is equal to $9$. Moreover, $k = 9$ happens if and only if $C(X,q)$ is the union of 6 lines on $Q$. This gives the following.

\begin{proposition} The number $d = k+1$ of ordinary double points of an irreducible cubic hypersurface in $\bbP^4$ is less than or equal to 10.
\end{proposition}

It follows from the proof of the previous proposition that the curve $C(X,q)$ is reducible if $d > 5$. The number of its irreducible components is equal to the fourth Betti number $b_4(X)$.\footnote{This fact was essentially known to G. Fano, see \cite{Fano1}.}
 The number $b_4(X)-1$  is called  the \emph{defect} $\textrm{def}(X)$ of $X$. The maximal number of 
linearly independent homology classes of exceptional curves in any small resolution of $X$ is equal to $d-\textrm{def}(X)$ (see \cite{FW}). Note that a small resolution of $X$ may not be a projective variety. The number of projective small resolutions of a nodal cubic threefold $X$ can be found in \cite{FW}.

In this article we will restrict ourselves with two  most interesting, in my view, cases when $d = 10$ or $d = 6$. We will start with the case $d = 10$, 

I am grateful to the referee and Alessandro Verra for many useful remarks that allowed me to improve the exposition of this paper. 
 
\section{Segre 10-nodal cubic primer}\label{S:2}
By a projective transformation, we can fix  the equation $a_2 = 0$ of the quadric $Q$ given in \eqref{cubic1}. The curve $C(X,q)$ has 9 singular points, this forces it to be the union of 9 lines on $Q$ as in the following picture:

\xy (-60,10)*{};(-30,-30)*{};
(0,0)*{};(30,0)*{}**\dir{-};
(0,-10)*{};(30,-10)*{}**\dir{-};
(0,-20)*{};(30,-20)*{}**\dir{-};
(5,5)*{};(5,-25)*{}**\dir{-};
(15,5)*{};(15,-25)*{}**\dir{-};
(25,5)*{};(25,-25)*{}**\dir{-};
(-10,-10)*{C(X,q):};
\endxy

Also, by an automorphism of the quadric, we can fix the curve $C(X,q)$. Two cubic forms defining  $F= 0$ and $F' = 0$ that  cut out $C(X,q)$ in $Q$, differ by $a_2l$, where $l$ is a linear form in $t_1,\ldots,t_4$. Applying the transformation 
$t_0\mapsto t_0+l$, we fix the equation $a_3 = 0$ of the cubic. This shows that 

\begin{proposition}
Two 10-nodal cubic hypersurfaces in $\bbP^4$ are projectively isomorphic. 
\end{proposition}

We choose one representative of the isomorphism class and denote it by $\sfS_3$. Any cubic threefold isomorphic to $\sfS_3$ is called a \emph{Segre cubic primal}. We refer to \cite{CAG}, 9.4.4 for many beautiful classical facts about such threefolds.  Here we add some more.  

 The polar quadric $V(\frac{\partial}{\partial t_0})$ of $X$ at a node $q$ cuts out in $X$ a surface of degree 6 with a point of multiplicity 4 at $q$ which is projected isomorphically outside $q$ onto the quadric $V(a_2)$ and passes through the nodes. This shows that the pre-image of each line component of $C(X,q)$ is a plane in $\sfS_3$ containing 4 nodes. Also it shows that no three nodes are collinear. Note that a cubic threefold containing a plane can be written in appropriate projective coordinates by equation $xQ_1+yQ_2 = 0$, where $Q_1$ and $Q_2$ are quadratic forms and the plane is given by $x = y = 0$. It follows immediately that the  points given by $x= y=Q_1 =Q_2 = 0$ are singular points of the threefold. Their  number is less than or equal to 4.  Thus we obtain that each node is contained in 6 planes and each plane contains 4 nodes. This gives the following.

\begin{proposition} The Segre cubic primal $\sfS_3$ contains 10 nodes and 15 planes that form an abstract configuration $(15_4,10_6).$
\end{proposition}

This synthetic argument belongs to Segre. Guido Castelnuovo \cite{Cast} and later, but independently,  H. W. Richmond \cite{Richmond}, were able to find the following 
$\frakS_6$-symmetric equations of the Segre cubic (see \cite{CAG}, 9.4.4).

\begin{equation}\label{segre2}
\sum_{i=0}^5t_i^3 = \sum_{i=0}^5t_i = 0.
\end{equation}
 
It is checked that the singular points form the $\frakS_6$-orbit of the point  
$[1,1,1,-1,-1,-1]$ and the planes form the $\frakS_6$-orbit  of the plane 
 $t_0+t_1 = 0, t_2+t_3= 0, t_4+t_5 = 0.$

In his paper  Castelnuovo studies linear systems of complexes of lines in $\bbP^4$, i.e. linear systems of hyperpane sections of the Grassmannian variety $G_1(\bbP^4)$ in its 
Pl\"ucker embedding in $\bbP^9$ (see \cite{CAG}, 10.2). A 3-dimensional linear system of such complexes defines a rational map
$$f:\bbP^3\dasharrow \bbP^4$$
which is given by pfaffians of principal $4\times 4$-matrices of a skew-symmetric matrix of size $5\times 5$ whose entries are linear forms in 4 variables. 
There are five points of indeterminacy  $q_1,\ldots,q_5$ of this map corresponding to skew-symmetric matrices of rank 2. The linear system defining the map $
f$ is equal to the linear system $|2h-q_1-\cdots-q_5|$ of quadrics passing through the points $q_1,\ldots,q_5$, where $h$ is the class of a plane in $\bbP^3$. If one chooses projective coordinates such that 
the points $q_i$ become $[1,0,0,0],\ldots,[0,0,0,1], [1,1,1,1]$, then the quadrics from the linear system acquire equations of the following type
$$\sum_{0\le i<j\le 3}a_{ij}u_iu_j = 0,$$
where $\sum a_{ij} = 0.$
The 5-dimensional vector space of such quadrics form the fundamental irreducible representation of the group $\frakS_6$. This shows that the image of the 
map $f$ admits a $\frakS_6$-symmetry. Castelnuovo and Richmond find a special basis in the linear system of quadrics to show that the image of the map $f$ can be given by the equations \eqref{segre2}. The images of the 10 lines 
$\overline{q_i,q_j}$ are the ten nodes of the cubic. The images $\Pi_{i,j,k}$ of the ten planes $\overline{q_i,q_j,q_k}$ 
and the images $\Pi_i$ of the five exceptional divisors $E_i$  blown-up from the points $q_i$ are the 15 planes of the cubic. 

In an earlier work of P. Joubert \cite{Joubert}, the equations \eqref{segre2} appear as the relations between 
certain six polynomials in roots of a general equation of degree 6. It shows that $\sfS_3$ should be considered as 
the set of ordered sets of 6 points in $\bbP^1$ modulo projective equivalence. Later A. Coble \cite{Coble1} made it more 
precise by proving that the geometric invariant quotient of $(\bbP^1)^6$ by the group $\SL(2)$ is isomorphic 
to the Segre cubic 
$$\sfS_3\cong \sfP_1^6:= (\bbP^1)^6/\!/\SL_2.$$

The rational map $f:\bbP^3\da \sfS_3$ can be extended to a regular map $\tilde{f}:X\to \sfS_3$, where $X\to \bbP^3$ is the 
composition of the blow-up of the five points $q_1,\ldots,q_5$ and the blow-up of the proper transforms of the lines 
$\overline{q_i,q_j}$. 

\beq\label{kap1}
\xymatrix{\bar{M}_{0,6}\ar[dr]\ar[rr]&&\bbP^3\ar[dl]^f\\
&\sfS_3&}.
\eeq

In modern times, the variety $X$  appears as the special case of M. Kapranov's realization  of the Knudsen-Mumford moduli space $\overline{\calM}_{0,6}$ of stable rational curves with $6$ marked ordered points \cite{Kapranov}. For any $n\ge 4$,  one chooses $n-1$ points in general linear position in $\bbP^{n-3}$, then start blowing up the points, then the proper transforms of lines joining two points, then proper transforms of planes joining 3 points, and so on. The result is isomorphic to the Knudsen-Mumford moduli space $\overline{\calM}_{0,n}$ of stable rational curves with $n$ marked ordered points. If one chooses a general point $q$ in $\bbP^n$ and passes through it and the points $q_1,\ldots,q_{n+2}$ the unique normal rational 
curve $C(q)$ of degree $n$, then one finds $n$ points on $C(q)\cong \bbP^1$, namely, the points $q_1,\ldots,q_n$ and the point $q$. 

It follows from  Kapranov's realization of  the variety $\overline{M}_{0,n}$ as a Chow quotient that we have the following commutative triangle of  regular maps

$$
\xymatrix{\bar{M}_{0,n}\ar[rr]^{\pi}\ar[dr]&&\bbP^{n-3}\ar[dl]_{f_n}\\
&\sfP_1^{n}:= (\bbP^1)^n/\!/\SL(2)&},
$$
If $n = 6$, the image of $f_n$ is the Segre cubic $\sfS_3$. 

In the case when $n = 2g+2$ is even, the geometric invariant theory quotient 
$\sfP_1^{2g+2}:= (\bbP^1)^{2g+2}/\!/\SL(2)$ is isomorphic to a compactification of the moduli space of hyperelliptic curves of genus $g$ together with a full 2-level structure, i.e. a choice of a standard symplectic basis in the group of 2-torsion divisor classes. In \cite{Coble2} A. Coble shows that the map $f_{2g+2}:\bbP^{2g-1}\da \sfP_1^{2g+2}$ is given by the linear system
\begin{equation}\label{coble3}
|gh-(g-1)(q_1+\cdots+q_{2g+1})|.
\end{equation}
It maps $\bbP^{2g-1}$ to the projective space $\bbP^{N-1}$, where $N = \binom{2g}{g}-\binom{2g}{2g-2}$.
 The image of this 
map is  isomorphic to $\sfP_1^{2g+2}$. 

Also, Coble shows that the projection of $\sfP_1^{2g+2}$ from a general point $p$ defines a degree 2 map onto the Kummer variety 
$\Kum(\Jac(C_p))$ associated with the Jacobian variety of the hyperelliptic curve $C_p$ corresponding to the point $p$. 
It lies in $\bbP^{2^g-1}$ and embeds there by the map $\Jac(C_p)\to \Kum(\Jac(C_p)) \to \bbP^{2^g-1}$ given by the linear system $|2\Theta|$, where $\Theta$ is the theta divisor of the Jacobian. The locus of singular points of the hypersurfaces from the linear system $|gh-(g-1)(q_1+\cdots+q_{2g+1})-q|$, where $f_{2g+2}(q) = p$, is the \emph{Weddle variety} $W_g$. It maps birationally onto the Kummer variety of $C_p$. We refer for a modern exposition of Coble's results to \cite{Dolg} and C. Kumar's paper \cite{Kumar}. Kumar calls the variety 
$\sfP_1^{2g+2}$ the \emph{generalized Segre variety}. One finds in Kumar's paper a nice relationship between the generalized Segre variety and the theory of vector bundles on hyperelliptic curves.

Note that Segre himself was aware of the relationship between the cubic $\sfS_3$ and the Kummer quartic surface  associated to curves of genus 2. In fact, he shows that the projection of $\sfS_3$ to $\bbP^3$ from its nonsingular point is the quartic Kummer surface $\Kum(\Jac(C_p)))$. He also shows that the set of nodes of quadrics in $\bbP^3$ passing through the points $q_1,\ldots,q_5$ and the additional point $q$ is the Weddle quartic surface.

A nice relationship between the generalized Segre variety $\sfP_1^{2g+2}$ and the theory of stable rank 2 vector bundles on not-necessary hyperelliptic curves was recently studied by A. Alzati and M. Bolognesi \cite{Alzati}.

Let 
$C_g$ be a non-hyperelliptic smooth projective curve of genus $g\ge 2$ and let  $\SU_{C_g}(2)$ be the moduli space of semi-stable rank 2 bundles on $C_g$ with trivial determinant. Alzati and Bolognesi prove that there exists a rational map 
$$\SU_{C_g}(2)\dasharrow \bbP^g$$
whose fibers are birationally isomorphic to $\sfP_1^{2g+2}$. If $g = 3$, they are isomorphic to $\sfS_3$.

Finally note that, according to \cite{Fink1}, the Segre cubic primal has 1024 small resolutions with 13 isomorphism classes, among them are 332  projective varieties which are divided into 6 isomorphism classes. Since $C(\sfS_3,q)$ has 6 irreducible components, the rank of the Picard group of a projective small resolution is equal to 5.
 
\section{Cubic threefolds with 6 nodes}
Let $X$ be a cubic hypersurface in $\bbP^4$ with six nodes. Choosing one of the nodes of $X$, we can get the equation of $X$ as in \eqref{cubic1}.  Let $Q = V(a_2)$ and $F = V(a_3)$ be the quadric and a cubic surface in the hyperplane $t_0 = 0$ defined by the coefficients $a_2$ and $a_3$ of the equation.

\begin{proposition}\label{nondeg} Assume that any five of the  nodes of $X$ span $\bbP^4$. Then $C(X,q) = Q\cap F$ is the union of two rational curves of degree 3  intersecting transversally at 5 points. 
\end{proposition} 

\begin{proof} We know that  remaining 5 nodes  of $X$ are projected to the double points $p_1,\ldots,p_5$ of the curve 
$C(X,q)$. By assumption, any four of them span $\bbP^3$. It implies that no four of the points $p_i$ lie on a conic or on a line contained in $Q$, and hence the curve $C(X,q)$ has no irreducible components of degree $\le 2$. Then an easy computation with the formula for the arithmetic genus of $C$ gives that $C$ consists of two components $\gamma$ and $\gamma'$ of degree 3.  Since the quadric $Q$ is nonsingular (otherwise $q$ is not an ordinary double point), the curves are curves of types $(2,1)$ and $(1,2)$ intersecting at the points $p_1,\ldots,p_5$. These points are the projections of the five nodes of $X$ and if two of them coincide,  together with $q$, they would lie on a line. This contradicts our assumption.
\end{proof}

 The following example shows that the condition on the six points is necessary for $C(X,q)$  to be the union of two curves of degree 3.

\begin{example}  Let $Q = V(a_2)$ be a nonsingular quadric and $C$ be  a curve on $Q$ equal to the union of a nonsingular conic $C_1$ and an irreducible curve $C_2$ of type $(2,2)$  with a node. Let $H= V(l)$ be a plane intersecting $Q$ along  $C_1$ and let $Q' = V(q)$  be a quadric intersecting $Q$ along  $C_2$.  Pick up a general plane $H' = V(l')$ such that $F = V(a_2l'+ql)$ is a nonsingular cubic. Then the cubic threefold given by equation \eqref{cubic1} with $a_2,a_3$ as above, has 6 nodes with 5 of them spanning $\bbP^3$ (four of them are projected to the intersection points $C_1\cap C_2$). 

Let $X$ be any 6-nodal cubic such that five of its nodes $q_1,\ldots,q_5$ span a hyperplane $H$. The intersection $H\cap X$ is a cubic surface with 5 nodes, hence it must be reducible. It is easy to see that this implies that four of the nodes are coplanar. In particular, any five nodes of $X$ span a hyperplane. 
\end{example}

We say that a 6-nodal cubic threefold is \emph{nondegenerate} if any subset of five nodes span $\bbP^4$. Thus Proposition \ref{nondeg} applies, and we obtain that the projection from any node gives an equation \eqref{cubic1}, where $V(a_2)\cap V(a_3)$ is the union of two rational cubics intersecting at 5 points.

Note that any 5-nodal complete intersection of a quadric and a cubic surfaces in $\bbP^3$ defines a 6-nodal cubic threefold, and the case from Proposition \ref{nondeg} is the only case which gives a nondegenerate 6-nodal cubic threefold.

We shall use the following well-known fact about cubic surfaces several times, so better let us record it.

\begin{lemma}\label{fact} Let $F$ be a smooth cubic surface and $\gamma_1,\gamma_2$ be two rational smooth curves of degree 3 on $F$ such that $\gamma_1+\gamma_2\in |-2K_F|$. Then the set of 27 lines on $F$ is divided into three disjoint sets:
\begin{itemize}
\item 6 skew lines that do not intersect $\gamma_1$ but intersect $\gamma_2$ with multiplicity 2,
\item 6 skew lines that do not intersect $\gamma_2$ but intersect $\gamma_1$ with multiplicity 2,
\item 15 lines that intersect both $\gamma_1$ and $\gamma_2$ at one point. 
\end{itemize}
\end{lemma}

\begin{proof} We have $\gamma_i\cdot K_F = -3$, hence $\gamma_i^2 = 1$ and the linear system $|\gamma_i|$ defines a birational morphism $\pi_i:F\to \bbP^2$. Each of these morphisms blows down six lines that together form a double-six. In the standard (geometric) basis 
$(e_0,e_1,\ldots,e_6)$  in $\Pic(F)$ defined by $\pi_1$, we have $\gamma_1\sim e_0$ and $\gamma_2\in |5e_0-2(e_1+\cdots+e_6)|$. The double-six is represented by the classes $e_i$ and 
$2e_0-(e_1+\cdots+e_6)+e_i, i = 1,\ldots,6$. Each line in the class $e_i$ is blown down by $\pi_1$ and intersects $\gamma_2$ with multiplicity 2. The remaining lines are represented by the classes $e_0-e_i-e_j$ which intersect $\gamma_1$ and $\gamma_2$ with multiplicity 1. 
\end{proof}

 Let $\Bis(C_i)$ be the surface of bisecant lines of $C_i$. It is naturally isomorphic to the symmetric product of $C_i$, and hence to $\bbP^2$. In the Pl\"ucker embedding of the Grassmann variety $G_1(\bbP^3) \subset \bbP^5$, it is isomorphic to a Veronese surface. Let $\ell_1,\ldots,\ell_6$ (resp. $\ell_1',\ldots,\ell_6$) be the set of lines in $\Bis(C_1)$ (resp. $\Bis(C_2)$)  corresponding to a double-six of lines on $F$ via the previous lemma.  For any $x\in F$, 
there exists a  unique bisecant line of $C_1$ (resp. $C_2$) passing through $x$. This defines two maps $\pi_1:F\to \Bis(C_1), \pi_2:F\to \Bis(C_2)$ that blows down the lines
$\ell_1,\ldots,\ell_6$ (resp. $\ell_1',\ldots,\ell_6'$).

\begin{proposition}\label{nonsing} Let $X$ be a nondegenerate 6-nodal cubic hypersurface and let $C(X,q) = Q\cap F$ be the fundamental curve associated to $q$.  Let $S$ be the blow-up of $Q$ at  the five nodes $q_1,\ldots,q_5$ of $C(X,q)$. Then $S$ is isomorphic to a nonsingular cubic surface.
\end{proposition}

\begin{proof}  The linear system of curves $|D| = |\calO_Q(2)-p_1-\cdots-p_5|$ of bidegree $(2,2)$ containing the five nodes $p_1,\ldots,p_5$ of $C(X,q)$ is 3-dimensional and defines a regular map $f:S\to S'$ onto a cubic surface $S' \subset |D|^* \cong \bbP^3$. Since $X$ is nondegenerate, the images of the 5 exceptional curves $E(q_i)$, the 10 lines each passing through one of the points $p_i$ (obviously, no two of them lie on one line), 10 conics passing through three of the points $q_i$, and the curves $C_1,C_2$ are the 27 lines on $S'$. This implies that $S'\cong S$ is a nonsingular cubic surface. 
\end{proof}

Let is add some remarks to the previous construction. The lines $\ell_1 = f(C_1)$ and $\ell_2 = f(C_2)$ are skew lines on $S'$, and the images of the exceptional curves $E(q_i)$ is the set of 5 skew lines that intersect 
both $\ell_1$ and $\ell_2$. The composition 
$f' = \sigma\circ f^{-1}:S'\to Q \cong \bbP^1\times \bbP^1$  is given by the two pencils  of conics cut out by the pencils of planes $\calP_1,\calP_2$ in $\bbP^3$ containing $\ell_1,\ell_2$, respectively. This map factors through an isomorphism $f^{-1}:S'\to S$ which is inverse of the morphism $f:S\to S'$.

For any $x\in S'$ there exists a unique line in $\bbP^3$ intersecting $\ell_1$ and $\ell_2$. It is obvious, if $x\not\in \ell_1\cup \ell_2$. If $x\in \ell_1$, we choose the line contained in the plane spanned by $\ell_2$ and $x$ and in the plane containing $\ell_1$ and tangent to $S'$ at $x$. Assigning to $x$ the intersection points of this line with $\ell_1$ and $\ell_2$, we obtain that the surface $S'$, and hence $S$, is isomorphic to the irreducible surface $S''$ in $G_1(\bbP^3)$ of lines in $\bbP^3$ that intersects both components of $C(X,q)$ The lines passing through the singular points of $C(X,q)$ must be tangent to the quadric at these points. Such lines are 5 lines on $S''$ that correspond to the skew lines intersecting $\ell_1$ and $\ell_2$.

Segre proves  that any nondegenerate 6-nodal cubic hypersurface can be projectively generated. This means that one can find three projectively equivalent nets of hyperplanes.
$$H(\lambda)_j:= \sum_{i=0}^4a_i^{(j)}(\lambda)t_i = 0, \quad j = 0,1,2,$$ 
such that 
$$X = \{x\in \bbP^4:x\in H_1(\lambda)\cap H_2(\lambda)\cap H_3(\lambda)\  \textrm{for some $\lambda$}\}.$$
Let us rewrite these equations in the following form
$$\lambda_0l_{0j}(t)+\lambda_1l_{1j}(t)+\lambda_2l_{2j}(x) = 0, \ j = 0,1,2,$$
where $l_{ij}(t)$ are linear forms in variables $t_0,t_1,t_2,t_3,t_4$. Then
\beq\label{det1}
X = \{x\in \bbP^4:\det (l_{ij}(x)) = 0\}.
\eeq
The six nodes of $X$ are the points $x$ such that $\rank (l_{ij}(x)) = 1.$

We see from formula \eqref{det1} that  a projective generation gives a determinantal representation of $X$. Conversely, the determinantal representation defines a projective generation.   

Segre's proof is rather cumbersome and I had a difficulty to follow it. A modern proof was given by B. Hassett and Yu. Tschinkel \cite{Hassett}. They deduce a determinantal representation of $X$ from a determinantal representation of a certain cubic surface associated to $X$. Let us reproduce a modified version of their proof that, in my opinion, is more straightforward and constructive. 

\begin{theorem} Let $X$ be a nondegenerate 6-nodal cubic hypersurface in $\bbP^4$. Then $X$ is isomorphic to the hypersurface $V(\det(A))$, where $A$ is a $3\times 3$-matrix with linear form in coordinates on $\bbP^4$. 
\end{theorem}

\begin{proof}
A normal cubic surface has only double rational points as its singularities. Each cubic surface without singular point of type $E_6$  is determinantal. This means that there is an embedding of $\bbP^3$ in the projective space  $\bbP^8$ of $3\times 3$ matrices such that the pre-image of the determinantal cubic hypersurface $\sfD_3$ in $\bbP^8$ is equal to $X$. 
This was proved first by L. Cremona in 1868 (with a gap related to the 
assumption on the singularities). C. Segre  had filled the gap in 1906. We refer to \cite{CAG}, 9.3 for the details.

So, our cubic surface $F = V(a_3)$, being nonsingular by our choice of projective coordinates, admits a determinantal representation. Let us recall its construction. We assume that  $F$ is a smooth cubic surface in the projective space $|W|$ of lines in a 
linear vector space $W$ of dimension 4. A determinantal equation of $F$ is defined by a choice of a linear system $|\gamma_1|$ of curves of degree 3 with $\gamma_1^2 = 1$. Let $|\gamma_2| = |-2K_F-\gamma_1|$ be represented by a smooth rational curve $\gamma_2$ with $\gamma_2^2 = 1$. The pair of smooth curves $(\gamma_1,\gamma_2)$ is a pair from Lemma \ref{fact}. Let $\pi_i:F\to 
|\gamma_i|^*$ be the corresponding  birational morphisms. The birational map 
$\pi_2\circ \pi_1^{-1}:|\gamma_1|^*\dasharrow |\gamma_2|^*$ is defined by the linear system of 
curves of degree 5 with double points at the points $p_i = \pi_1(\ell_i)$, where $(\ell_1,\ldots,\ell_6)$ is the sixer of lines blown down by $\pi_1$. Consider the natural map defined by adding the divisors 
\beq\label{schur}
|\gamma_1|\times |\gamma_2| \to |6e_0-2e_1-\cdots-2e_6| = |-2K_F| \cong |\calO_{|W|}(2)|.
\eeq
 The image of this map is a hyperplane in $|\calO_{|W|}(2)|$ orthogonal to a quadric $\calQ^*$ in the dual space $|W^\vee|$. It is  the dual quadric  of  the Schur  quadric $\calQ$  associated to the double-six  of lines blown-down by $\pi_1$ and $\pi_2$ (see \cite{CAG}, 9.1.3). Composing \eqref{schur} with a linear function defined by $\calQ^*$, we can identify the plane $|\gamma_2|$ with the plane $|\gamma_1|^*$, the dual plane of $|\gamma_1|$. Let $|\gamma_1| = |U|$ for some 3-dimensional linear space $U$. Then the map 
$\pi_1\times \pi_2:F\to |\gamma_1|^*\times |\gamma_2|^*$ can be identified with the  linear map
$$j:F\to |U^*|\times |U| \hookrightarrow |U^*\otimes U| =  |\End(U)| \cong \bbP^8.$$
Let $\sfD_3$ be the determinantal hypersurface in  $|\End(U)|$. It is a cubic hypersurface with singular locus equal to $|U^*|\times |U|$ (it  parameterizes endomorphisms of rank 1). The determinantal representation of $F$ is defined by the embedding $|W|\hookrightarrow |\End(U^*)| \cong |\End(U)|^*$ such that the pre-image of the determinantal hypersurface is equal to $S$ and the map $\pi_1$ (resp. $\pi_2$) is defined by taking the  kernel of the corresponding endomorphism (resp. its transpose) of $U^*$.
All of this is well-known and can be found in \cite{CAG}. 
 
Now let $X$ be a nondegenerate $6$-nodal cubic hypersurface given by equation \eqref{cubic1} and $C_1+C_2 = C(X,q)$ be the associated curve with respect to a node $q$. It follows from the above discussion that the linear systems $|C_1|$ and $|C_2|$ define two determinantal representation of $F$, each is obtained from another by taking the transpose of the matrix. 

Choose a basis in the linear space $U = H^0(F,\calO_F(\gamma_1))$ and the dual basis in $U^*$, then we can identify $\End(U)$ with the space of $3\times 3$-matrices and the determinantal representation of $F$ gives a matrix $B = (b_{ij})$ whose entries are linear forms in $t_1,\ldots,t_4$ such that 
$$F = V(\det B).$$
 We are looking for a matrix $\tilde{B} = (\tilde{b}_{ij})$ whose entries are linear forms in $t_0,\ldots,t_4$ such that 
$$X = V(\det \tilde{B}).$$ 
If we plug in $t_0 = 0$ in the entries of $\tilde{B}$, we should obtain a matrix equal to $B$, up to a scalar multiple. This shows that 
$\tilde{B} = t_0A+B$, where $A$ is a constant matrix. Write $A = [A_1 A_2 A_3]$ and $B = [B_1 B_2 B_3]$ as the collection of its columns. The usual formula for the determinant of the sum of the matrices shows that 
\beq\label{matrix}
\det \tilde{B} = t_0^3\det A+t_0^2 (\det [B_1 A_2 A_3]+\ldots)+t_0(\det [B_1 B_2 A_3]+\ldots)+\det B.
\eeq
To make this expression equal to $t_0a_2+a_3$ from \eqref{cubic1}, we have to take $A$ with rank equal to 1. So we may assume that the columns of $A$ are equal to some nonzero vector $\bfv = (\alpha_0,\alpha_1,\alpha_2)$. 

Let $\pi_1:F\to \bbP^2 = |U|$ and $\pi_2:F\to |U^*|$ be the two maps  defined by the right and the left kernels of the matrix $B$. Let $\ell = \pi(\gamma)$ and $\ell' = \pi'(\gamma')$. The curve $\ell$ is a line in $\bbP^2$, the curve $\ell'$ is a line in the dual plane.  Observe that, for any $y = [t_1,\ldots,t_4]\in F$, the adjugate matrix $\adj B$ of $B$ is of rank 1. Thus, for any $x\in F$, the equations $\det [B_1(x) B_2(x) \bfv] = 0, 
\det [B_1(x) \bfv B_3(x)] = 0, \det [\bfv B_2(x) B_3(x)] = 0$ with unknown vector $\bfv$, are the equations of the same line $\ell(x)$ in $\bbP^2$ which we consider as a point in the dual plane. When $x$ runs $C_1$, the image $\pi_1(C_1)$ is a line in the plane, hence the set of lines $\ell(x), x\in C_1,$ is a line in the dual plane. If we take it to be equal to the line equal to $\pi_2(C_2)$, we obtain that the coefficient at $t_0$ in \eqref{matrix} is equal to zero for any $x\in C_1\cup C_2$. Thus the quadric $Q'$ defined by this coefficient coincides with the quadric $Q$, and we are done. \end{proof}

\begin{remark}\label{R3.6} Suppose equation \eqref{cubic1} of a nodal cubic threefold can be brought to the form  $\det A(t) = 0$. Then, plugging in $t_0 = 0$, we obtain a determinantal representation of the cubic surface $F = V(a_3)$. This shows that $F$ has at most rational double points of type different from  $E_6$. Also, since the discriminant variety $\sfD_3$ has the double locus of degree 6, we obtain that the singular locus of $X$ is either of dimension $\ge 1$, or consists of isolated singular points whose Milnor numbers add up to 6.
\end{remark}

Let $X$ be a nondegenerate 6-nodal cubic threefold with nodes $q_1,\ldots,q_6$. The linear system of cubics $|\calO_{\bbP^4}(3)-2q_1-\cdots-2q_6|$ with double points at $q_1,\ldots,q_6$ defines a rational map $f:\bbP^4\dasharrow \sfS_3$ to the Segre cubic primal $\sfS_3$ in $\bbP^4$. Its fibers are quartic rational normal curves passing through the nodes. In Kapranov's  realization of $\calM_{0,7}$ this corresponds to the composition of the 
projection $\calM_{0,7}\to \calM_{0,6}$ and the map $\calM_{0,6}\to \sfS_3$ from \eqref{kap1}. This shows that $X$ is birationally isomorphic  to the pre-image of a hyperplane section of $\sfS_3$. Let $X'$ be the blow-up of $X$ at the nodes, followed by the blow-up of the proper transforms of lines joining two nodes. Then $f$ extends to a regular map $X'\to S$, where $S$ is a hyperplane section of $\sfS_3$. If we use equation \eqref{segre2} of  $\sfS_3$, then the additional equation 
$\sum_{i=0}a_it_i = 0$ defines a cubic surface $S$ given by Cremona's hexahedral equations (see \cite{CAG}, 9.4.3). By Theorem 9.4.8 from loc.cit., if $S$ is nonsingular, these equations determine uniquely an ordered double-six of lines on $S$. Conversely, a choice of an ordered double-six of lines defines Cremona's hexahedral equations. 

It is an obvious guess that the cubic surface $S$ is isomorphic to the cubic surface from Proposition \ref{nondeg}. To see this, we consider the blow-up $X'$ of $X$ at any of its singular point $q$. The exceptional divisor is identified with the quadric $Q_i$ containing $C(X,q_i)$.  The pre-image of the linear system $|\calO_{\bbP^4}(3)-2(q_1+\cdots+q_6)|$  to $X'$ restricted to the exceptional divisor  $E(q)$ consists of quadrics through the 5 points on $Q_i$ corresponding to the lines joining $q_i$ with other nodes of $X$. It maps $E(q_i)$ to the hyperplane section of $\sfS_3$ corresponding to $X$.

We can easily see the double-six of lines on $S$ identified with the blow-up of $Q_i$ at the singular  points of $C(X,q_i)$. It consists of the images of the ten lines on $Q$ passing through the singular points and the curves $C_1$ and $C_2$. The linear systems $|\gamma_1$ and $|\gamma_2|$ defining a determinantal representation of $S$ are the images of the curves of bi-degree $(3,1)$ and $(1,3)$ passing through the singular points of $C(X,q)$. The order on the 6 nodes of $X$ defines an order on the set of singular points of $C(X,q)$, hence an order of the lines in a sixer, hence a basis of the Picard group of the cubic surface $S$.

Finally, let us look at the moduli space of nondegenerate 6-nodal cubic hypersurfaces. By a projective transformation we can  fix the nodes, to assume that their coordinates form the reference points in $\bbP^4$. Then the space of cubics with singular points at the reference points consists of cubics with equations
$$\sum_{0\le i < j<k \le 4}a_{ijk}t_it_jt_k = 0,$$
where the coefficients $a_{ijk}$ satisfy 
$$\sum a_{0jk} = 0,\ \sum a_{i1k} = 0,\ldots, \sum a_{ij4} = 0.$$
The dimension of the projective  space $|V|$ of such cubics is equal to 4. The permutation group $\frakS_6$ acts linearly on $V$ via its natural $5$-dimensional irreducible permutation. By above, the linear system of such cubics map $\bbP^4$ onto the Segre cubic primal $\sfS_3$. Thus its dual space is identified with the space of hyperplane sections of $\sfS_3$. The action of $\frakS_6$ agrees with the natural action of $\frakS_6$ on $\sfS_3$. We know that the orbit of a hyperplane section of $\sfS_3$ corresponds to the moduli space of cubic surfaces together with an unordered double-six. It is the cover of degree 6 of the moduli space of cubic surface.  It is known that such variety is rational (see \cite{Bauer}).
This gives us the following result.

\begin{theorem} The moduli space of nondegenerate 6-nodal cubic threefolds is naturally birationally isomorphic to the moduli space of cubic surfaces together with an unordered double-six of lines. It is a rational variety of dimension 4.
\end{theorem}

\begin{remark} Let $X \subset |W| \cong \bbP^4$ and $W\to \End(U)$ be a linear map defined by a determinantal representation  of $X$. Consider the nondegenerate bilinear form on $\End(U)$ defined by the trace. It allows one to identify $\End(U)$ with its dual space $\End(U)^\vee$. The orthogonal space $W^\perp$ is of dimension 4, and the linear embedding $W^\perp \hookrightarrow \End(U)$ defines a determinantal representation of the cubic surface $S' = |W^\perp|\cap \sfD_3$. It is proven in \cite{Hassett} that this cubic surface is isomorphic to the cubic surface $S$.

Conversely, starting with a determinantal representation of a nonsingular cubic $S \subset |V|$ defined by 
a linear map $V\to \End(U)$, we obtain a determinantal 6-nodal cubic $X$ by taking the intersection of $|V^\perp|$ with the determinantal hypersurface $\sfD_3$. This is the approach to determinantal representations of 6-nodal cubic threefolds taken in \cite{Hassett}. It is analogous to the earlier construction of A. Beauville and R. Donagi \cite{Beauville} that uses pfaffian hypersurface to pair  K3 surfaces of genus 8 and 4-dimensional pfaffian cubic hypersurfaces. 
\end{remark}

The number of small resolutions of a nondegenerate 6-nodal cubic threefold is equal to 64. Among them there are two projective resolutions \cite{FW}. Note that a degenerate irreducible 6-nodal cubic does not admit a small projective resolution.

\section{The Fano surface of lines}\label{S:1}
Corrado Segre also studied the  surface $F(\sfS_3)$ of lines in $\sfS_3$. Later on, Gino Fano wrote two papers which establish the basic facts about the surface of lines in any nodal cubic threefold \cite{Fano1}, \cite{Fano2}. For a modern exposition of some of Fano's result see  \cite{Altman}). Here we remind some known facts about the  Fano surface $F(X)$ of lines in a nodal cubic threefold, nowadays called the \emph{Fano surface} of $X$. In the Pl\"ucker embedding, the surface $F(X) \subset G_1(\bbP^4)$ is a locally complete intersection projectively normal surface of degree 45 canonically embedded in the Pl\"ucker space $\bbP^9$. Its class in $H^4(G_1(\bbP^4),\bbZ)\cong \bbZ^2$ with respect to the basis given by the Schubert cycles $\sigma_1$ of lines intersecting a general line and $\sigma_2$ of lines in a general hyperplane  is equal to $(18,27)$. 

The Fano surface $F(X)$  is smooth at any point representing a line that does not pass through a singular point of $X$. The union of lines passing through a node $q\in X$ is projected from $q$ to the curve $C(X,q)$. In particular, $F(X)$ is  a non-normal surface. If  the number $d$ of nodes is less than or equal to 5, it is an irreducible surface birationally isomorphic to the surface of bisecant lines of the associated curve $C(X,q)$. Starting from $d = 6$, the curve $C$ become reducible, and $F(X)$ becomes reducible too. 

\smallskip
Let us start with the Segre cubic primal $\sfS_3$. 
Under the map $f:\bbP^3\dasharrow \sfS_3$ given by the linear system of quadrics through 5 points $p_1,\ldots,p_5$, the image  of a line $\ell$ containing the point $p_i$ is a line in $\sfS_3$. The set $D_i$ of such lines is isomorphic to a del Pezzo surface of degree 5, the blow-up of 4 points in the plane (the points are of course the lines joining $p_i$ with $p_j, j\ne  i$). This gives us five isomorphic irreducible components of $F(\sfS_3)$. The image of a twisted cubic passing through the points $p_1,\ldots,p_5$ is also a line in $\sfS_3$. The set of such lines is also isomorphic to a del Pezzo surface of degree 5. To see this, we use the Kapranov's realization of the projection map 
$\overline{\calM}_{0,6}\to \overline{\calM}_{0,5}$. So we have 15 components isomorphic to $\bbP^2$, they are lines containing in a plane in $\sfS_3$. Each of these planes are the images of a conic in the plane $ \la p_i,p_j,p_k\ra$ through the points $p_i,p_j,p_k$ or the lines  in one of the exceptional divisor $E_i$ over the point $q_i$.
Each line intersects five planes. For example, a line from $D_1$  intersects the plane corresponding to the exceptional divisor  $E_i$ and four planes 
$\la p_a,p_b,p_c\ra, 1\not\in \{a,b,c\}.$ So, we see that $F(\sfS_3)$ is highly reducible, it consists of 6 components isomorphic to a del Pezzo surface of degree 5, and 15 components isomorphic to the projective planes. In the Pl\"ucker embedding of $G_1(\bbP^4)$ they are anti-canonically embedded surfaces of degree 5, and the planes.

The group $\frakS_6$ of automorphisms of $\sfS_3$ acts transitively on the set of 6 components $D_i$ of the Fano surface. The stabilizer subgroup is isomorphic to the group of automorphisms of $D_i$. It also acts transitively on the 15 plane components of the Fano surface. The stabilizer group is isomorphic to $2^3\rtimes \frakS_3$ of order 48.

\begin{remark} The description of conic bundle structures on the Segre cubic primal can be also found in  \cite{Gwena}.
\end{remark}

\bigskip
Next we consider the case when $X$ is a nondegenerate 6-nodal cubic threefold. 

Since $X$ has 6 nodes, the Fano surface is singular along the curve parameterizing lines passing through a node.  Fix a node $q$ of $X$ and let $C(X,q) = C_1+C_2 = Q\cap F$ be the associated curve of degree 6. Each line passing through $q$ is projected to a point on $C(X,q)$. Since $C(X,q)$ consists of two curves of degree 3, we see that the lines through $q$ are  parameterized by $C(X,q)$ and their union consists of two cubic cones intersecting along 5 common lines.

 Let $\Bis(C)$ be the set of lines in $\bbP^3$ that intersect $C(X,q)$ at $\ge 2$ points counting with multiplicity. For any two points $x,y\in C(X,q)$, let $\ell = \overline{x,y}$ be the line spanned by $x,y$ or tangent to $Q = V(a_2)$ if $x =y$. Suppose $\ell$ intersects $C$ at two points $x,y$. The linear system of cubics through $C(X,q)$ maps $\bbP^3$ to $X$ and the image of $\ell$ is a line $\ell_{x,y}$ on $X$. If $\ell_x,\ell_y$ are the lines through $q$ which are projected to $x,y$, then $\ell_{x,y}$ is the residual line of the intersection of the plane spanned by $\ell_x,\ell_y$ with $X$ (if $x = y$ we take the tangent plane of $X$ along the line $\ell_x$). If $\ell$ intersects $C(X,q)$ at three point, then it must belong to one of the two rulings of $Q$. Let $C_i$ be the component of $C(X,q)$ such that this ruling intersects $C_i$ at one point $z$. Then we assign to it the line $\ell_z$ passing through $q$. For any line $\ell$ in $X$, let $\Pi(\ell)$ be the unique plane containing $q$ and cutting $X$ in 3 lines. If $q\not\in \ell$, then $\Pi(\ell)$ is spanned by $\ell$ and $q$. Otherwise, $\Pi(\ell)$ is the unique plane cutting $X$ along three lines passing through $q$. The projection of $\Pi(\ell)$ to $\bbP^3$ from the point $q$ is a line from $\Bis(C(X,q))$. It is contained in $Q$ if $\Pi(\ell)\cap X$ consists of three lines passing through $q$. This defines an isomorphism
 $$ \Bis(C(X,q))\cong F(X).$$
Obviously, $\Bis(C(X,q))$ consists of three irreducible components, two of which are isomorphic to $P_i = \Bis(C_i)$. The third component $P_3$ consists of lines intersecting both $C_1$ and $C_2$.

For any general point in $\bbP^3$ there exists a unique bisecant line of a normal rational cubic. Since a general line intersects $X$ at three points, we see that the Schubert cycle $\sigma_1$  intersects the components $P_1$ and $P_2$ of $F(X)$ at three points. A general hyperplane $H$ in $\bbP^4$ intersects $X$ along a nonsingular cubic surface. It also intersects each of the cubic cones of lines passing through $q$ along a cubic rational curve $\gamma_i$.  The curve $\gamma_i$ can be identified (under the projection map) with a component $C_i$ of $C(X,q)$. Using Lemma \ref{fact}, we see that this defines a set of 6 skew lines, each intersecting one of the curves $\gamma_i$ with multiplicity 2. These lines correspond to 6 bisecants of $C_i$. This shows that the bidegree of the surface $\Bis(C_i)$ in $G_1(\bbP^4)$ is equal to $(3,6)$, it is isomorphic to $\bbP^2$ embedded by the third Veronese map.

The third component $P_3$ of $F(X)$ must be of bidegree $(12,15)$. As above, 15 corresponds to the 15 lines in a general hyperplane that intersect $\gamma_1$ and $\gamma_2$, and 12 corresponds to the fact that through each point $x$ on the intersection of a general line with $X$  passes 4 lines whose projection to $\bbP^3$ intersects $C_1$ and $C_2$ at one point. 

It follows from Proposition \ref{nonsing} that $P_3$ is isomorphic to the cubic surface $S$ isomorphic to the blow-up of $Q$ at the singular points of $C(X,q)$. Thus we have a map
$$\phi:P_1\sqcup P_2\sqcup P_3 \cong \bbP^2\sqcup \bbP^2\sqcup S \to F(X)$$
that coincides with the normalization map.

Note that $P_1\cap P_2$ consists of 10 points corresponding to lines $\overline{q_i,q_j}$. The intersection $P_i\cap P_3$ consists of 5 lines corresponding to the lines through $q_i$ intersecting $C_i$ at some other point. 

\begin{remark} It follows from the above that the isomorphism class of the cubic surface isomorphic to the blow-up of $Q$ at singular points of $C(X,q)$ is independent of a choice of a node. 
\end{remark}

\begin{remark} The following nice observation is due to A. Verra. Take any line $\ell$ on $X$ and consider its image $f(\ell)$ under the map $f:\bbP^4\to \bbP^4$ given by the linear system of cubics through the nodes of $X$. As we saw in above, the image of $X$ is isomorphic to the cubic surface $S$. The three components of $F(X)$ are distinguished by the following different kinds of the curves $f(\ell)$. The image of a line from the components $P_1$ or $P_2$ is a twisted cubics $\gamma_i$ such that $|\gamma_i|$ gives one of the two blowing down structures of $S$ defined by a double-six of lines. The image of a line from the component $P_3$  is cut out by a tangent plane section of the cubic surface. 
\end{remark}

 One can easily describe the Fano surface $F(X)$ of lines on $X$ using a determinant representation $X = V(\det A)$ of $X$. Let $q_1,\ldots,q_6$ be the nodes of $X$ and $X^{\ns} = X\setminus \{q_1,\ldots,q_6\}$. The nodes correspond to the points at which the matrix $X$ is of rank 1. The  right and the left kernel of $A$ define two regular map $\frakr:X^{\ns}\to \bbP^2$ and $\frakl:X^{\ns}\to \bbP^2$. The fiber $\frakr^{-1}(p)$ (resp. $\frakl^{-1}(p)$) consists of the points $x\in X'$ such that the right kernel (resp. the left kernel) of the matrix $A(x)$ is equal to $p$. The restriction of these maps to the hyperplane $t_0 = 0$ in equation \eqref{cubic1} define the two blowing-down structures on the cubic surface $F = V(a_3)$ defined by the linear systems $|C_1|, |C_2|$. These are two irreducible components of $F(X)$ isomorphic to $\bbP^2$.  At each node $X$ contains two  scrolls $\calS_1, \calS_2$ of lines passing through $q_i$. We may assume that $\frakr(\calS_1)$ and $\frakl(\calS_2)$ are lines. For any pair of lines 
$\ell_i \subset\calS_i$, the plane spanned by $\ell_1$ and $\ell_2$ intersects $X$ at the union of $\ell_1,\ell_2$ and a third line $\ell$. It belongs to the third irreducible component of $F(X)$. The pair of points  
$(y_1,y_2) = (\ell\cap\ell_1,\ell\cap\ell)$ corresponds to a point on the cubic surface 
$S$ associated to $X$ via Proposition \ref{nondeg}. 

Since the components $P_1,P_2$ of $F(X)$ do not depend on the determinant representation
of $X$, we obtain the following.

\begin{corollary} A nondegenerate 6-nodal cubic threefold admits two non-equivalent determinant representations, one is obtained from another by taking the transpose of the matrix.
\end{corollary}

\section{Conic bundles}
Let $X$ be a $d$-nodal cubic hypersurface in $\bbP^4$ and $\ell$ be a line on $X$. We assume that $\ell$ is not contained in any plane on $X$. The projection from $\ell$ defines a conic bundle structure on the blow-up $X'$ of $X$ along $\ell$.  The discriminant curve $\Delta_\ell$ of this bundle is of degree 5 and has $\ge d$ ordinary nodes or cusps. To see this, one takes a general hyperplane section $F$ of $X$ that contains $\ell$. It is know that a line in a nonsingular cubic surface intersects ten more lines forming 5 pairs of lines that are coplanar with $\ell$. The restriction of the projection to $F$ defines a conic bundle on $F$ with 5 singular fibers. The projection from $\ell$ defines a conic bundle on $F$ equal to the pre-image of the conic bundle on $X'$ over a line. This shows that the degree of the discriminant curve is equal to 5. 

Let $\Gamma_\ell$ be the curve of lines in $X$ intersecting $\ell$. For any line $\ell'\in \Gamma$, we have the line $\ell''$ such that $\ell,\ell',\ell''$ are coplanar. This defines an involution 
$\iota:\Gamma_\ell\to \Gamma_\ell$, and the quotient by this involution is isomorphic to $\Delta_\ell$. If $X$ is a nonsingular, G. Fano showed in \cite{Fano2} that $\Gamma$ is a nonsingular curve of genus 11 and degree 15 in $G_1(\bbP^4)$  (see a modern proof in \cite{CG}). 

We start with the case $d = 10$. A line on $X$ which is not contained in a plane in $X$ belongs to one of the six del Pezzo components $D_i$. We assume also that $\ell$ is a general line from $D_i$. Since $\frakS_6$ permutes these components, we may assume that  $\ell\in D_1$. Consider the projection map $\textrm{Bl}_{\sfS_3}(\ell)\to \bbP^2$, 
where $\textrm{Bl}_{\sfS_3}(\ell)$ is the blow-up of the line $\ell$.  We consider $\sfS_3$ as the image of the rational map $\bbP^3\dasharrow \sfS_3$ given by quadrics through 5 points $p_1,\ldots,p_5$. The curve $\Gamma_\ell$ consists of 10 components. Four of the components   are conics $K_i$ of lines in $D_i, i\ne 1,6,$ represented by lines in $\bbP^3$ intersecting $\ell$. Another four components are lines $L_{ijk}$ represented by conics in the planes $\Pi_{ijk}, 1\not\in \{i,j,k\}$. Also we have the conic $K_5$ of lines in $D_6$ represented by rational normal cubics through $p_1,\ldots,p_5$ intersecting $\ell$ and the pencil $L_1$ of lines in the exceptional divisor $E(p_1)$ passing through the point corresponding to the line $\ell$. The involution $\iota:\Gamma_\ell\to \Gamma_\ell$ pairs 
the conic $K_i$ with the line $L_{abc}, i\not \in \{a,b,c\}$ and  the conic $K_5$ with the line $L_1$.  The following picture shows a component of $\Delta_\ell$ represented by the pair
$K_2$ and $L_{345}$.

\xy (-50,-15)*{};(-50,25)*{};
(0,0)*{};(37,17.5)*{}**\dir{-};
(50,0)*{};(15,17.5)*{}**\dir{-};
(2,-3)*{p_1};(48,-3)*{p_2};(13,-5)*{p_3};(38,-5)*{p_4};(25,18)*{p_5};
(39,18)*{\ell};(13,18)*{\ell'};(25,-6)*{C};
(2,1)*{\bullet};(48,1)*{\bullet};
(16,-5)*{\bullet};(34.5,-5)*{\bullet};(25,15)*{\bullet};
(25,3)*\cir<35pt>{};
\endxy

Next we consider the case of a nondegenerate $6$-nodal cubic threefold. 

 Let $F$ be a nonsingular cubic surface obtained by blowing up $6$ points $x_1,\ldots,x_6$ in the plane. Let $(e_0,e_1,\ldots,e_6)$ be the corresponding geometric basis of $\Pic(F)$. Let $|\gamma_1| = |e_0|$ and $|\gamma_2| = |5e_0-2(e_1+\ldots+e_6)$ as in Lemma \ref{fact}. We divided all 27 lines on $F$ in three disjoint subsets. The first set of 6 lines intersecting $\gamma_1$ will be called lines of type $(0,2)$, they are represented by the exceptional curves $E(x_i)$. The second set of 6 lines will be called lines of type $(2,0)$, they are represented by the conics through 5 points. The remaining 15 lines will be called lines of type $(1,1)$, they are represented by lines joining a pair of the six points.
 
Let $H$ be a hyperplane in $\bbP^4$ that cuts out $X$ along a nonsingular surface $F$. We fix a node $q$ to identify the linear system of hyperplanes with the linear system of cubics  in $H$ containing the fundamental curve $C(X,q) = C_1+C_2$. Thus, we can divide all lines in $F$ in three sets as above.

Take a general line  $\ell$ in the component $P_1$ of $F(X)$. Let $\Gamma_\ell$ be the curve of lines in $X$ intersecting $\ell$. Since no curve from $P_1$ different from $\ell$ intersects $\ell$, we see that $\Gamma_1$ consists of two components $\Gamma_\ell(2)$ and $\Gamma_\ell(3)$ of curves from $P_2$ or $P_3$ intersecting $\ell$. Take another general curve $\gamma'$ from $P_1$ and let $F$ be a cubic surface in $X$ that contains $\ell$ and $\ell'$. We see that $\ell,\ell'$ are lines on $F$ of types $(2,0)$. We may represent then by two exceptional curves $E(q_i),E(q_j)$ of the blow-up of 6 points. The number of lines on $F$ of  type $(0,2)$ intersecting $\ell$ and $\ell'$ is equal to 4, they are represented by conics passing through 5 points including $p_i$ and $p_j$. This implies that $\Gamma_\ell(2)$ has self-intersection in $P_1\cong \bbP^2$ equal to 4. Similarly, we compute the self-intersection of $\Gamma_\ell(3)$. It is equal to 1. We also obtain that  $\Gamma_\ell(2)\cdot \Gamma_\ell(3)  = 2$. Thus $\Gamma_\ell(2)$ is a conic on $\bbP^2$ and $\Gamma_\ell(3)$ is a line. In the Pl\"ucker embedding of $G_1(\bbP^4)$, they are curves of degrees 6 and 3.\footnote{To get expected degree 15 one has to add six lines here that come from a choice of one ruling in each exceptional divisor of the resolution $X'\to X$.} The involution $\iota_\Gamma$ switches the two components. The discriminant curve $\Delta_\ell$ is an irreducible rational curve of degree 5 with 6 nodes.

Next, we take $\ell$ to be a general line in the component $P_3$. The curve $\Gamma_\ell$ consists of three components $\Gamma_\ell(1),\Gamma_\ell(2), \Gamma_\ell(3)$ of lines from $P_1,P_2,P_3$, respectively, intersecting $\ell$. By similar argument as above, we find that two disjoint lines of type $(1,1)$ on $F$ are intersected by one line of type $(0,2)$, one line of type $(2,0)$ and three lines of type $(1,1)$. This intersection matrix of these curves is equal to 
$$\begin{pmatrix}1&5&3\\
5&1&3\\
3&3&3\end{pmatrix}.$$
The curves $\Gamma_\ell(1)$ and $\Gamma_\ell(2)$ are rational curves of degree 3 on the cubic surface $P_3$, and the curve $\Gamma_\ell(3)$ is an elliptic curve of degree 3 on $P_3$. In the Pl\"ucker embedding of $G_1(\bbP^4)$ the curves are of degrees $3,3$ and $6$, respectively, and their sum is a hyperplane section equal to $-3K_{P_3}$. The involution $\iota$ switches the first two components and defines a fixed-point free involution on the third component. The discriminant curve $\Delta_\ell$ consists of the union of a conic and a cubic intersecting at 6 points.

\begin{remark} It follows that, for any general line from the component $P_3$, there exists a quadric singular along $\ell$ that contains the six nodes of $X$. Note that, counting parameters, we see that it is not true if we replace $\ell$ with a general line in $\bbP^4$. 

Also, $X$ contains an elliptic curve parameterizing the singular points of the conics corresponding to the degree 3 component of the discriminant curve $\Delta_\ell$, where $\ell\in P_3$. This elliptic curve contains the six nodes of $X$. When $\ell$ varies in $P_3$, we get a family of elliptic curves (or their degenerations) parameterized by the cubic surface $P_3$. Is it an elliptic fibration on the blow-up of $X$ at 6 nodes?
\end{remark}

\begin{remark} It is known that the intermediate Jacobian of a smooth cubic threefold is a principally polarized abelian variety of dimension 5. 
It is also known that it is isomorphic to the Prym variety associated to the discriminant curve of degree 5 of the conic bundle on $X$ defined by the projection from a line on $X$ (see, for example, \cite{Beauville1}). In the case when $X$ is a nondegenerate 6-nodal cubic threefold, the intermediate Jacobian  $J(X)$ degenerates to a 5-dimensional algebraic torus which is isomorphic to the generalized Prym variety $\textrm{Prym}(\tilde{\Delta}/\Delta)$, where $\Delta$ is the discriminant curve of the conic bundle on $X$ defined by the projection from a line on $X$ (as was recently shown in a joint work of  S. Casalaina-Martin, S. Grushevsky, K. Hulek and R. Laza, this isomorphism does not depend on a choice of a line). 

The projection from a line $\ell$ on $X$ corresponding to the bisecant of $C(X,q)$ joining two points $x,y$ on different component of $C(X,q)$ can be viewed as the conic bundle structure of the blow-up $\tilde{\bbP}^3$ of $\bbP^3$ at the reducible  curve $C(X,q)\cup \ell$ of degree 7 and arithmetic genus 5 which is given by the linear system of cubics through the curve. In the case of a connected smooth curve of degree 7 of genus 5 in $\bbP^3$ this conic bundle was considered by V. Iskovskikh (it is a Fano variety with the Picard number equal to 2, No 9 in the list of such varieties that can be found in \cite{Iskovskikh}). 
\end{remark}

 \section{6-nodal cubic threefolds and nonsingular cubic fourfolds}
 
Let $Y$ be a nonsingular cubic hypersurface in $\bbP^5$. Suppose $Y$ contains a  normal cubic scroll $T$ spanning a hyperplane $H$ in $\bbP^5$. It is isomorphic to the blow-up of $\bbP^2$ at one point $p$ and it is embedded in $\bbP^4$ by the linear system of 
conics through the point $p$ (see \cite{CAG}, 8,1,1). Choosing the point $p$ to be $[1,0,0]$ and the basis of the linear system in the form 
$(x_2^2,x_2x_3,x_1x_2,x_1x_3,x_1^2)$, we can 
express the  equations of the scroll as
$$\rank\begin{pmatrix}t_0&t_1&t_2\\
t_2&t_3&t_4\end{pmatrix} \le 1.$$
A cubic hypersurface in $\bbP^4$ containing $T$ can be given by the equation 
$$l_1(t_1 t_4-t_2 t_3)+l_2(t_2^2-t_0 t_4)+l_3(t_0 t_3-t_1 t_2) = 0,$$
For some linear forms $l_i$ in variables $t_0,\ldots,t_4$. For appropriate $l_1,l_2,l_3$, we find the intersection $X = H\cap X$. Clearly the previous equation gives a determinantal representation of $X$. As was explained in Remark \ref{R3.6}, we expect that $X$ is a nondegenerate 6-nodal cubic threefold.

Assume now that $Y$ is a nonsingular cubic 4-fold that contains a cubic scroll $T$ whose linear span intersects $Y$ along a nondegenerate 6-nodal cubic 3-fold $X$. We know that   the degree of the curve in $G_1(\bbP^4)$ parameterizing the ruling of $T$  is equal to 3 
(\cite{CAG}, 10,4,1.) This implies that $T$ is the pre-image of a line under under the two rulings of $X$ parameterized by the components $P_1,P_2$ of $F(X)$. In the determinant representation of $X$, $T$ is the pre-image of a line under the left or the right kernel maps.
Assume that $T$ is the pre-image of a line in the component $P_1$. We know that 
$\dim H^4(X,\bbQ) = $2. Let $\calQ$ be a quadric in $\bbP^4$ that contains $T$ (it could be taken as one of the minors defining $T$), it intersects $X$ along the union of $T$ and another cubic scroll $T'$ corresponding to a line in the component $P_2$. 

Let
$$\xymatrix{&Z\ar[dl]^{\sigma}\ar[dr]^{\tau}&\\
Y&&F(Y)}
$$
Be the incidence correspondence of lines and points in $Y$, where $F(Y)$ is the Fano fourfold of lines in $Y$. It is known to be an irreducible holomorphic symplectic 4-fold. The Abel-Jacobi map of integral Hodge structures
$$\Phi:\tau_*\sigma^*:H^4(Y,\bbZ)[2] \to H^2(F(Y))$$
 defines an isomorphism of free abelian groups of rank 21:
$$H^{2,2}(Y)\cap H^4(Y,\bbZ) \to H^{11}(F(Y),\bbZ)\cap H^2(F(Y),\bbZ)).$$
 The group $H^2(F(Y),\bbZ)$ is equipped with the Beauville-Bogomolov quadratic form $q_{BB}$ of signature $(1,20)$. The isomorphism $\Phi$ is a compatible (with the change of the sign) with the cup-product on primitive cohomology of $H^4(Y,\bbZ)$ and the Beauville-Bogomolov quadratic form restricted to the primitive cohomology of $H^2(F(Y),\bbZ)$. Let 
$\sigma$ be the class of a hyperplane section of $F(Y)$ in its Pl\"ucker embedding in $G_1[\bbP^5)$. It is known that the degree of $F(Y)$ is equal to 36, so that 
$\sigma^4 = $36. However, $q_{BB}(h) = $6. 

For a nonsingular cubic 4-fold $Y$, we have 
$$\Pic(F(Y)) = H^2(F(Y),\bbZ)\cap H^{11}(F(Y),\bbZ) \cong \bbZ^r,$$
If $Y$ is general in the sense of moduli (the number of them is equal to 20), then $r = 1$. 
 Let $h$ be the class of a hyperplane section of $Y$. Then 
$[X] = h$ and $2h^2 = [T]+[T']$. This shows that $\Phi([T]),\Phi[T']$ belong to $\Pic(F(Y))$ and give $r \ge $2. We assume that $r = $2 and the classes 
$\tau = \Phi([T])$ and $\sigma$ freely generate $\Pic(F(Y))$. The quadratic lattice 
$(H^4(Y,\bbZ),\cup)$  has a basis $(h^2,[T])$ and is defined in this basis by the matrix
$\left(\begin{smallmatrix}3&6\\
3&7\end{smallmatrix}\right).$
The quadratic lattice $(\Pic(F(Y)),q_{BB})$ has a basis $(\sigma,\tau)$ and is defined in this basis by the matrix
$\left(\begin{smallmatrix}6&6\\
6&2\end{smallmatrix}\right).$

In their paper \cite{Hassett}, B. Hassett and Yu. Tschinkel compute the nef cone of $F(Y)$ in $\Pic(F(Y))_\bbR$ and find that it equals the dual cone of the cone
spanned by the classes $\alpha_1 = 7\sigma-3\tau$ and $\sigma+3\tau$. 

Recall that   $F(X)$ contains two planes $P_1,P_2$. Since, obviously,  $F(X)\subset F(Y)$, we see that $F(Y)$ contains two planes (embedded in the Pl\"ucker space of $G_1(\bbP^5)$ as  
Veronese surfaces spanning $\bbP^5$). 

We may perform Mukai flop $P_1$ (resp. $P_2$) to obtain new holomorphic symplectic manifolds $V_1,V_2$ (\cite{Mukai}). The birational maps between $F(Y)$ and $V_1$ (resp. $V_2$) identify their Picard groups but induce an orthogonal transformation $R_1$ (resp $R_2$) of their quadratic forms. 

The following beautiful result is proven in \cite{Hassett}.

\begin{theorem} There is an infinite sequence of flops
$$\cdots F_{-2}\dasharrow F_{-1}^\vee \dasharrow F(Y)\dasharrow F_1\dasharrow F_2\cdots$$
with isomorphisms $F_i\cong F_{i+2}$. The cone of effective divisors on $F(Y)$ can be expressed as the non-overlapping union of nef cones of the $F_i$'s. The birational pseudo-automorphism of $F(Y)$ defined by $F(Y)\dasharrow F_1\dasharrow F_2\cong F(Y)$ acts on $F(Y)$ by the matrix $\left(\begin{smallmatrix}-1&-2\\
6&11\end{smallmatrix}\right)$ in the basis $\sigma,\tau$. The following picture copied from \cite{Hassett}:

\begin{figure}[ht]
\begin{center}
\includegraphics[scale=.7]{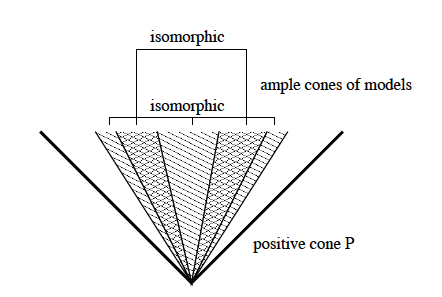}
%\caption{Dual Apollonian circle clusters (from \cite{Maxwell})}
\end{center}
\end{figure}
\end{theorem}

\bibliographystyle{plain}

\end{document}